\documentclass[11pt,reqno]{amsproc}
\usepackage[all]{xy}
\usepackage[colorlinks,urlcolor=blue,citecolor=blue,linkcolor=blue]{hyperref}
\usepackage{stix2,xcolor,graphicx,bm,anysize}
\usepackage{caption,subcaption,eqnarray}

\marginsize{15mm}{15mm}{6mm}{5mm}

\newtheorem{theorem}{Theorem}[section]

\theoremstyle{remark}

\newtheorem{remark}[theorem]{\it Remark}

\theoremstyle{definition}

\AtBeginDocument{\,\vskip-2cm}

\numberwithin{equation}{section}
\allowdisplaybreaks

\newcommand{\cdummy}{\cdot}
\newcommand{\mathd}{\mathrm{d}}

\newcommand{\tmop}[1]{\ensuremath{\operatorname{#1}}}
\newcommand{\tmtextit}[1]{\text{{\itshape{#1}}}}

\begin{document}
\title[Physical derivation of the coarea formula \& proof]{Physical derivation of the coarea formula and an elementary proof via
gradient flow}
\author{Shibo Liu\vspace{-1em}}
\dedicatory{Department of Mathematics and Systems Engineering, Florida Institute of Technology\\
Melbourne, FL 32901, USA}
\thanks{Emails: \texttt{\bfseries sliu@fit.edu} (S. Liu)}
\begin{abstract}
In this note, we derive an elementary version of the coarea formula by considering the mass of a solid body with density $g (x)$. Then we present an rigorous proof using the changing variable formula. To this end we construct the diffeomorphism $\Phi$ via the gradient flow and compute its Jacobian determinant via geometric method.
\end{abstract}
\maketitle

The coarea formula is a fundamental result in geometric measure theory. The
most general version of the formula can be found in the treatise
{\cite{federer1996general}} of Federer. To appreciate the formula, a lot of
prerequisites on measure theory and geometry are needed. Therefore, it seems
that an easier proof for a special (but general enough) version of the formula
is meaningful. Such a proof will make this important and useful formula more accessible
to mathematicians (including undergraduate and graduate students) not familiar
with geometric measure theory.

\begin{figure}[h]
\centering
\begin{subfigure}[b]{0.45\textwidth}
\includegraphics[clip,viewport=2 355 155 488]{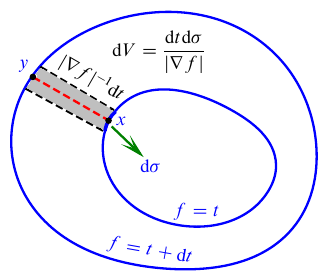}
\caption{Physical interpretation of coarea formula}
\label{fxc}
\end{subfigure}
\hfill
\begin{subfigure}[b]{0.45\textwidth}
\includegraphics[clip,viewport=7 635 155 763]{fig}
\caption{Decomposition of $f^{-1}[a,b]$ and $f^{-1}(t)$}
\label{fyc}
\end{subfigure}
\caption{}
\end{figure}

We start with an inspiring physical interpretation. For our function $f :
\mathbb{R}^n \rightarrow \mathbb{R}$, imagine $\Omega = f^{- 1} [a, b]$ as a
solid body in $\mathbb{R}^n$ with density $g (x)$ at $x \in \Omega$. Then, the
mass of $\Omega$ is
\begin{equation}
  M = \int_{\Omega} g (x) \mathd x \text{.} \label{e}
\end{equation}
Given $t \in [a, b]$, for $x \in f^{- 1} (t)$ we take a surface area element
$\mathd \sigma$ at $x$. Let $\mathd t$ be an infinitesimal increment of $t$
and $y$ be the intersection of the normal line of the level surface $f^{- 1}
(t)$ at $x$ with the nearby level surface $f^{- 1}  (t + \mathd t)$, see
Figure \ref{fxc}.

The first order Taylor expansion yields the following approximation
\begin{equation}
  \mathd t = f (y) - f (x) \approx \nabla f (x) \cdummy (y - x) \text{.}
  \label{e6}
\end{equation}
Observing that both $\nabla f (x)$ and $y - x$ are normal vectors of the level
surface $f^{- 1} (x)$ at $x$, they are parallel. Hence, approximately we have
\begin{equation}
  |y - x| = \frac{\mathd t}{| \nabla f (x) |} \text{.} \label{n}
\end{equation}
Since $\mathd \sigma$ and $\mathd t$ are infinitesimal, the volume $\mathd V$
and the mass $\mathd m$ of the small cylinder (illustrated as the shadowed
region in Figure \ref{fxc}) based on $\mathd \sigma$ and bounded between $f^{- 1}
(t)$ and $f^{- 1}  (t + \mathd t)$ (thus, the height is $|y - x|$) are given
by
\[ \mathd V = \frac{1}{| \nabla f (x) |} \mathd t \mathd \sigma \text{,
   \qquad} \mathd m = \frac{g (x)}{| \nabla f (x) |} \mathd t \mathd \sigma
   \text{.} \]
Integrating $\mathd m$ (with respect to $x$) along the level surface $f^{- 1}
(t)$, we obtain the mass of the region $f^{- 1} [t, t + \mathd t]$:
\[ \mathd M = \int_{f^{- 1} (t)} \mathd m = \int_{f^{- 1} (t)} \frac{g (x)}{|
   \nabla f (x) |} \mathd t \mathd \sigma = \mathd t \int_{f^{- 1} (t)}
   \frac{g (x)}{| \nabla f (x) |} \mathd \sigma \text{.} \]
Now, integrating $\mathd M$ with respect to $t \in [a, b]$, we see that the
mass of $\Omega = f^{- 1} [a, b]$ is
\[ M = \int_a^b \mathd M = \int_a^b \mathd t \int_{f^{- 1} (t)} \frac{g (x)}{|
   \nabla f (x) |} \mathd \sigma \text{.} \]
Comparing this with (\ref{e}), we conclude
\begin{equation}
  \int_{f^{- 1} [a, b]} g (x) \mathd x = \int_a^b \mathd t \int_{f^{- 1} (t)}
  \frac{g (x)}{| \nabla f (x) |} \mathd \sigma \text{.} \label{ec}
\end{equation}
This is an elementary version of the coarea formula.

\begin{theorem}[Coarea formula]
  \label{t7}Let $U$ be an open subset of $\mathbb{R}^n$, $f : U \rightarrow
  \mathbb{R}$ be a $C^2$ function satisfying
  \begin{enumerate}
    \item[$(f)$] $\Omega = f^{- 1} [a, b]$ is compact and free of critical
    points.
  \end{enumerate}
  Then \eqref{ec} holds for any continuous $g : \Omega \rightarrow
  \mathbb{R}$.
\end{theorem}

\begin{proof}
The condition $(f)$ ensures that $| \nabla f | \geq \varepsilon$ on $W$ for
  some $\varepsilon > 0$ and open set $W$ containing $\Omega$. On $W$ we may
  define a bounded $C^1$ vector field $\phi : W \rightarrow \mathbb{R}^n$ via
  \[ \phi (x) = | \nabla f (x) |^{- 2} \nabla f (x) \text{.} \]
  There are a maximal open subset $G$ of $\mathbb{R} \times W$, and a $C^1$
  map $\eta : G \rightarrow \mathbb{R}^n$ called the \emph{maximal flow} of $\phi$,
  satisfying
  \begin{equation}
    \partial_t \eta (t, p) = \phi (\eta (t, p)) \text{, \quad} \eta (a, p) = p
    \text{; \qquad for } (t, p) \in G \text{.} \label{de}
  \end{equation}
  For $p \in f^{- 1} (a)$ we write $x_p = \eta (\cdummy, p)$. Then $x_p$ is
  the solution of the initial value problem
  \[ x' = \phi (x) \text{, \qquad} x (a) = p \text{.} \]
  Let $\beta$ be the maximal real number such that $x_p$ is defined on $[a,
  \beta)$. Then for $t \in [a, \beta)$ we have
  \begin{align*}
    (f \circ x_p)' (t) & =  \nabla f (x_p (t)) \cdummy x_p' (t) = \nabla f
    (x_p (t)) \cdummy \phi (x_p (t))\\
    & =  \nabla f (x_p (t)) \cdummy \frac{\nabla f (x_p (t))}{| \nabla f
    (x_p (t)) |^2} = 1 \text{.}
  \end{align*}
  We claim that $\beta > b$. If $\beta \leq b$, then for $t \in [a, \beta)$ we
  have
  \begin{equation}
    f (x_p (t)) = f (p) + \int_a^t (f \circ x_p)' (\tau)  \hspace{0.17em}
    \mathd \tau =f(p)+(t-a)= t. \label{E}
  \end{equation}
  Thus $x_p (t) \in \Omega$, hence $x_p (t) \in W$. Now let $t_{\pm} \in [a,
  \beta)$, from
  \[
    |x_p (t_+) - x_p (t_-) |  =  \left| \int_{t_-}^{t_+} \phi (x_p (t))
    \mathd t \right| \leq |t_+ - t_- | \cdummy \sup_{q \in W} | \phi (q) |
  \]
  and the Cauchy criterion we see that as $t \nearrow \beta$, $x_p (t)$ has a
  limit $x_p (\beta)$. Letting $t \to \beta$ in (\ref{E}) we get $f (x_p
  (\beta)) = \beta \leq b$. Thus $x_p (\beta) \in W$, and the solution $x_p$
  can be extended beyond $\beta$, contradicting the maximality of $\beta$.

  Having verified $\beta > b$, by continuous dependence of solutions on
  initial values, there is $\varepsilon (p) > 0$ such that for $q \in
  B_{\varepsilon (p)} (p)$, $x_q = \eta (\cdummy, q)$ is defined at least on
  $[a, b]$ as well. Then,
  \[ A = W \cap \bigcup_{p \in f^{- 1} (a)} B_{\varepsilon (p)} (p) \]
  is an open neighborhood of $f^{- 1} (a)$. We have $[a, b] \times A \subset
  G$, because $(t,q)\in G$ if and only if $\eta$ is defined at $(t,q)$, that is $x_q$ is defined at $t$.

  By the theory of ordinary differential equations, the restriction $\eta :
  [a, b] \times A \rightarrow \mathbb{R}^n$ is a $C^1$ map with the
  properties:
  \begin{enumerate}
    \item[(a)] $\eta : [a, b] \times f^{- 1} (a) \rightarrow f^{- 1} [a, b]$
    is bijective;

    \item[(b)] \label{i1}for $(t, q) \in [a, b] \times f^{- 1} (a)$, $\eta (t,
    \cdummy)$ is a local diffeomorphism near $q$, thus its Jacobian matrix
    $\partial_p \eta = (\partial_{x^i} \eta^j)_{i, j = 1, \ldots, n}$ at $(t,
    q)$ is invertible.
  \end{enumerate}
  For simplicity we first consider the case that $f^{- 1} (a)$ is a
  \tmtextit{parametrized surface}, i.e., the image of a single
  $C^1$-parametrization; this is the case if $f^{- 1} (a)$ is diffeomorphic to
  the unit sphere $S^{n - 1}$.

Let $\varphi : U \rightarrow \mathbb{R}^n$
  be a $C^1$-\tmtextit{parametrization} of $f^{- 1} (a)$, namely $U$ is a
  Jordan measurable compact region in $\mathbb{R}^{n - 1}$, $\varphi$ is
  continuous, and $\varphi (U) = f^{- 1} (a)$, moreover $\varphi
  |_{U^{\circ}}$ is injective and $C^1$, $\tmop{rank} \varphi' (u) = n - 1$
  for $u \in U^{\circ}$, where $U^{\circ}$ is the interior of $U$. We define a
  map $\Phi : [a, b] \times U \rightarrow \mathbb{R}^n$ via
  \begin{equation}
    \Phi (t, u) = \eta (t, \varphi (u)) \text{,}
  \end{equation}
  that is, $\Phi = \eta \circ (1_{[a, b]} \times \varphi)$, see Figure \ref{pic1}.

\begin{figure}[h]
\centering
\includegraphics[width=\textwidth,clip,viewport=9 491 365 627]{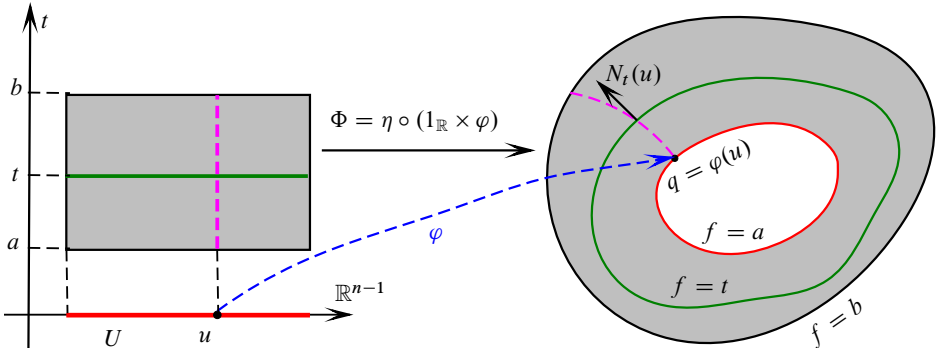}
\caption{Definition of the map $\Phi$}
\label{pic1}
\end{figure}

Using the above properties of $\eta$, it is easy to see that $\Phi ([a, b]
  \times U) = f^{- 1} [a, b]$, $\Phi |_{(a, b) \times U^{\circ}}$ is $C^1$ and
  $\Phi$ maps $(a, b) \times U^{\circ}$ injectively \tmtextit{into} $f^{- 1}
  [a, b]$. Moreover, using Property (b) and $\tmop{rank} \varphi' (u) = n - 1$
  we have
  \begin{equation}
    \tmop{rank} \partial_u \Phi (t, u) = \tmop{rank} [\partial_p \eta (t,
    \varphi (u)) \varphi' (u)] = n - 1 \text{.} \label{rk}
  \end{equation}
  Hence $\Phi (t, \cdummy) : U \rightarrow \mathbb{R}^n$ is a parametrization
  of $f^{- 1} (t)$ for $t \in [a, b]$, which gives rise to normal vector
  \[ N_t (u) = \left( \frac{\partial (\Phi^2, \ldots, \Phi^n)}{\partial (u^1,
     \ldots, u^{n - 1})}, \ldots, (- 1)^{n + 1} \frac{\partial (\Phi^1,
     \ldots, \Phi^{n - 1})}{\partial (u^1, \ldots, u^{n - 1})} \right) \]
  of $f^{- 1} (t)$ at $\Phi (t, u)$; $N_t (u) \neq 0$ because of (\ref{rk}).
  For $(t, u) \in (a, b) \times U^{\circ}$,
  \[
    \Phi' (t, u)  =  (\partial_t \Phi, \partial_u \Phi) = \left(
    \begin{array}{cccc}
      \partial_t \Phi^1 & \partial_{u^1} \Phi^1 & \cdots & \partial_{u^{n -
      1}} \Phi^1\\
      \partial_t \Phi^2 & \partial_{u^1} \Phi^2 & \cdots & \partial_{u^{n -
      1}} \Phi^2\\
      \vdots & \vdots &  & \vdots\\
      \partial_t \Phi^n & \partial_{u^1} \Phi^n & \cdots & \partial_{u^{n -
      1}} \Phi^n
    \end{array} \right) \text{.}
  \]
  Note that $\partial_t \Phi (t, u) = \partial_t \eta (t, \varphi (u))$ and,
  for $i \in \{1, \ldots, n\}$ the cofactor of $\partial_t \Phi^i$ in $\det
  \Phi' (t, u)$ is precisely the $i$-th component of $N_t (u)$. Expanding
  $\det \Phi' (t, u)$ along the first column and using the differential
  equation (\ref{de}), we deduce
  \begin{align}
    | \det \Phi' (t, u) | & =  | \partial_t \Phi (t, u) \cdummy N_t (u) | = |
    \partial_t \eta (t, \varphi (u)) \cdummy N_t (u) | \nonumber\\
    & =  \left| \frac{\nabla f (\Phi (t, u))}{| \nabla f (\Phi (t, u)) |^2}
    \cdummy N_t (u) \right| = \frac{|N_t (u) |}{| \nabla f (\Phi (t, u)) |}
    \text{,}  \label{f}
  \end{align}
  where the last equality follows from the observation that both $\nabla f
  (\Phi (t, u))$ and $N_t (u)$ are normal vectors of the surface $f^{- 1} (t)$
  at $\Phi (t, u)$, hence the absolute value of their dot product equals the
  product of their norms.

  From (\ref{f}) we have $\det \Phi' (t, u) \neq 0$. Hence the bijection
  \[ \Phi : (a, b) \times U^{\circ} \rightarrow \Phi ((a, b) \times U^{\circ})
  \]
  is a diffeomorphism. Applying a version of the changing variable formula
  (see for example {\cite[Theorem 1.1]{MR3711061}}, which only requires the
  transformation to be a diffeomorphism between the \tmtextit{interiors} of
  the regions of integration) and then the Fubini theorem, we deduce
  \begin{align}
    \int_{f^{- 1} [a, b]} g (x) \mathd x & =  \int_{\Phi ([a, b] \times U)} g
    (x) \mathd x \nonumber\\
    & =  \int_{[a, b] \times U} g (\Phi (t, u)) | \det \Phi' (t, u) | \mathd
    u \mathd t \nonumber\\
    & =  \int_a^b \mathd t \int_U g (\Phi (t, u)) \frac{|N_t (u) |}{| \nabla
    f (\Phi (t, u)) |} \mathd u \nonumber\\
    & =  \int_a^b \mathd t \int_{f^{- 1} (t)} \frac{g (x)}{| \nabla f (x) |}
    \mathd \sigma \text{,}  \label{fti}
  \end{align}
  where the last equality is due to the formula for computing surface integral
  over $f^{- 1} (t)$ via its parametrization $u \mapsto \Phi (t, u)$.

  If $f^{- 1} (a)$ is the union of several \tmtextit{interiorly disjoint}
  parametrized surfaces $\varphi_i (U_i)$, $i = 1, \ldots, \ell$, then for
  $\Phi_i : [a, b] \times U_i \rightarrow \mathbb{R}^n$ defined via $\Phi_i
  (t, u) = \eta (t, \varphi_i (u))$, we have
  \[ f^{- 1} [a, b] = \bigcup_{i = 1}^{\ell} \Phi_i  ([a, b] \times U_i)
     \text{,} \]
  see Figure \ref{fyc}. Generalizing the computation in (\ref{fti}) yields
  \begin{align*}
    \int_{f^{- 1} [a, b]} g (x) \mathd x & =  \sum_{i = 1}^{\ell}
    \int_{\Phi_i  ([a, b] \times U_i)} g (x) \mathd x\\
    & =  \sum_{i = 1}^{\ell} \int_a^b \mathd t \int_{\Phi_i (t, U_i)}
    \frac{g (x)}{| \nabla f (x) |} \mathd \sigma\\
    & =  \int_a^b \left( \sum_{i = 1}^{\ell} \int_{\Phi_i (t, U_i)} \frac{g
    (x)}{| \nabla f (x) |} \mathd \sigma \right) \mathd t\\
    & =  \int_a^b \mathd t \int_{f^{- 1} (t)} \frac{g (x)}{| \nabla f (x) |}
    \mathd \sigma
  \end{align*}
  because the interiorly disjoint union of $\Phi_i (t, U_i)$ is exactly $f^{-
  1} (t)$, see Figure \ref{fyc}.
\end{proof}

\begin{remark}
  The gradient flow has been used in many areas of mathematics. For example,
  it was used to prove the deformation lemma in critical point theory
  ({\cite[{\textsection}1.2]{MR1400007}}).

  In {\cite[pp. 504--505]{ding24}}, using similar idea the coarea formula
  (\ref{ec}) is proved for the case $n = 2$. As can be seen from our proof of
  Theorem \ref{t7}, the computation of the Jacobian determinant $\det \Phi' (t,
  u)$ in (\ref{f}) is crucial. In {\cite[Page 505]{ding24}}, because only $n = 2$ is
  concerned this is done by using cross product $\Phi_t \times \Phi_u$. A novelty of our proof is the clever computation of $\det \Phi'
  (t, u)$ in (\ref{f}) for arbitrary dimension $n \geq 2$. Therefore, with
  some new insight our note generalizes the first half of {\cite{ding24}}. As
  far as we know, the physical derivation presented before Theorem \ref{t7} is
  original.

  After the research has completed, we noticed that in {\cite[Lemma
  1]{Ding2018}}, the higher dimensional coarea formula (\ref{ec}) is stated.
  However, many crucial details are missing in the proof there. In particular,
  the author mentioned ``transformation of variables'', but he didn't present
  the computation of the Jacobian determinant. Therefore, our detailed proof of the
  coarea formula fills this gap.
\end{remark}

\end{document}